\documentclass[a4paper,12pt]{article}
\usepackage{amsmath}
\usepackage{amssymb}
\usepackage{setspace}
\usepackage{fullpage}
\usepackage{bbm}
\usepackage{bm}
\usepackage{amsthm}
\usepackage{pdfsync}
\usepackage{hyperref}
\newtheorem{theorem}{Theorem}[section]
\newtheorem{lemma}[theorem]{Lemma}

\newtheorem{corollary}[theorem]{Corollary}

\newtheorem{proposition}[theorem]{Proposition}
\theoremstyle{definition}
\usepackage[normalem]{ulem}
\usepackage{xcolor}

\newtheorem{definition}{Definition}

\newtheorem{remark}{Remark}

\newcommand{\stkout}[1]{\ifmmode\text{\sout{\ensuremath{#1}}}\else\sout{#1}\fi}
\def\PG{\mathrm{PG}} 
\def\F{\mathbb{F}}

\def\Aut{\mathrm{Aut}}
\def\PGammaL{\mathrm{P}\Gamma\mathrm{L}}
\def\GammaL{\Gamma\mathrm{L}}

\def\PGL{\mathrm{PGL}} 
\def\GL{\mathrm{GL}}
\def\Fq{\mathbb{F}_q}
\def\Fn{\mathbb{F}_{q^n}}
\def\a{\alpha}
\def\b{\beta}
\def\Tr{\mathrm{Tr}}
\def\tr{\mathrm{tr}}

\def\EQ1{[EQUIVALENCE 1]}
\def\EQ2{[EQUIVALENCE 2]}

\def\V{\mathrm{V}}
\def\rank{\mathrm{rank}}
\def\End{\mathrm{End}}
\def\la{\langle}
\def\ra{\rangle}
\def\ran{\rangle_{q^n}}
\def\L{\mathcal{L}}

\def\C{\mathcal{C}}
\def\D{\mathcal{D}}
\def\H{\mathcal{H}}
\def\U{\mathcal{U}}

\def\E{\End_{\Fq}(\Fn)}

\newcommand{\pt}[1]{\langle #1 \rangle}

\title{Rank-metric codes, linear sets, and their duality}
\author{John Sheekey \and Geertrui Van de Voorde \thanks{This author is a postdoctoral fellow of the Research Foundation Flanders (FWO -- Vlaanderen).}}
\begin{document}
\maketitle
\begin{abstract} 
In this paper we investigate connections between linear sets and subspaces of linear maps. We give a geometric interpretation of the results of \cite[Section 5]{Sheekey} on linear sets on a projective line. We extend this to linear sets in arbitrary dimension, giving the connection between two constructions for linear sets defined in \cite{Lunardon}. Finally, we then exploit this connection by using the MacWilliams identities to obtain information about the possible weight distribution of a linear set of rank $n$ on a projective line $\PG(1,q^n)$.

\end{abstract}

{\bf Keywords:} MRD code, weight distribution, linear set, scattered with respect to hyperplanes


\section{Introduction}

Linear sets are important in finite geometry due to their usefulness in constructing and characterising geometrical objects, for example blocking sets and finite semifields; see \cite{olga} for an in-depth treatment of this subject. Scattered linear sets are of particular interest, we refer to \cite{LavScat} for a complete survey.


In this paper we investigate connections between linear sets and subspaces of linear maps. An $\Fq$-subspace $\C$ of $\Fq$-linear maps from $\Fn$ to itself leads naturally to a linear set $\Omega(\C)$ in $\PG(n-1,q^n)$. If such a subspace is in fact a $k$-dimensional $\Fn$-subspace, it also naturally leads to a linear set $\Gamma_\C$ in $\PG(k-1,q^n)$. These two seemingly distinct objects were considered in \cite{Lunardon}, \cite{Sheekey}. In \cite{Sheekey} an algebraic connection was shown between these linear sets in the special case of $k=2$, giving a correspondence between $2$-dimensional $\Fn$-linear {\it MRD codes} and {\it scattered linear sets} of rank $n$ on the projective line $\PG(1,q^n)$. This connection has lead to further investigations in \cite{CsajMarPol,CsajMarPolZan,CsajMarZul}.

In this paper we will give a geometric interpretation of the correspondence of \cite{Sheekey}, and extend it to all dimensions. Specifically, we show that
\[
\Gamma_\C \cong \Sigma/\Omega(\C)^\perp,
\]
where $\Sigma$ is a subgeometry, and show how the distribution of the ranks of the linear maps in $\Omega(\C)$ determine the distribution of weights of the points of $\Gamma_\C$.

Furthermore we characterise the linear sets of rank $n$ in $\PG(k-1,q^n)$ defined by $k$-dimensional $\Fn$-linear MRD-codes. We show that they are precisely those linear sets which are {\it scattered with respect to hyperplanes}, a concept which we introduce and is stronger than a linear set being scattered.

Finally we exploit Delsarte's theory of duality for subspaces of linear maps (in particular the rank-metric MacWilliams identities) to obtain information about the possible weight distribution of linear sets of rank $n$ on a projective line $\PG(1,q^n)$.


\section{Preliminaries}

\subsection{Linear sets}
\subsubsection{Definition}
Throughout this paper, we denote the finite field with $q$ elements by $\Fq$ and a $k$-dimensional vector space over $\F_q$ by $V(k,q)$.
The $(k - 1)$-dimensional projective space corresponding to $V(k, q)$ is denoted by $\PG(k-1, q)$. Points in $\PG(k-1,q)$ correspond to one-dimensional subspaces of $V(k,q)$, and $(j-1)$-dimensional subspaces in $\PG(k-1,q)$ correspond to $j$-dimensional subspaces of $V(k,q)$. In general, if $V$ is a vector space, then $\PG(V)$ denotes the correponding projective space.

If $\Fq$ contains a subfield $\F_{q'}$, then we call a subset $U$ of $V(k,q)$ an $\F_{q'}$-subspace if $U$ is closed under addition, and by scalar multiplication by elements of $\F_{q'}$; i.e. for all $u,v\in U$, $\lambda\in \F_{q'}$, we have that $u+v\in U$ and $\lambda u\in U$.

For any subset $U$ of $V(k,q)$ we denote by $U^{\ast}$ the set of nonzero elements of $U$.

\begin{definition}
Suppose $\F_{q'}$ is a subfield of $\Fq$. An {\it $\F_{q'}$-linear set of rank $s$} in $\PG(k-1,q)$ is a set
\[
L(U) := \{\la u\ra_{\Fq} : u\in U^{\ast}\}
\]
for some $\F_{q'}$-subspace $U$ of $V(k,q)$ with $\dim_{\F_{q'}}(U)=s$. Here $\la u\ra_{\Fq}$ denotes the projective point in $\PG(k-1,q)$, corresponding to the vector $u$, where the notation reflects the fact that all ${\Fq}$-multiples of $u$ define the same projective point. 
When $u=(u_0,\ldots,u_k)$ is a vector in $V(k,q)$, then $\la u\ra_{\Fq}=\la (u_0,\ldots,u_k)\ra_{\Fq}$ will be denoted as $(u_0,\ldots,u_k)_{\Fq}$.

\end{definition}

In this paper we will be mostly concerned with $\Fq$-linear sets of rank $n$ in $\PG(k-1,q^n)$. Linear sets of rank $n$ are of interest as they are the largest linear sets in terms of rank which do not necessarily meet every hyperplane. Particular attention will be paid to the case where $k=2$, that is, linear sets on a projective line. The following representation will be used throughout; for the first part of this statement, see also \cite[Lemma 7]{Lunardon}.

\begin{lemma}\label{linsetpol}
Suppose $L$ is an $\Fq$-linear set of rank $n$ in $\PG(k-1,q^n)$. Then there exist $\Fq$-linear maps $f_i:\Fn\rightarrow\Fn$ such that
\[
L = \{(f_1(x),\ldots,f_k(x))_{\Fn}:x\in \Fn^{\ast}\},
\]
and $\ker(f_1)\cap\cdots\cap\ker(f_k)=\{0\}$. If $\dim\la L\ra=k-1$, then the maps $\{f_1,\ldots,f_k\}$ are linearly independent over $\Fn$.

Vice versa, if $f_i:\Fn\rightarrow\Fn$ are $\Fq$-linear maps with $\ker(f_1)\cap\cdots\cap\ker(f_k)=\{0\}$, then $L = \{(f_1(x),\ldots,f_k(x))_{\Fn}:x\in \Fn^{\ast}\}$ is an $\F_q$-linear set of rank $n$ in $\PG(k-1,q^n)$.
\end{lemma}

\begin{proof} As $L$ is an $\Fq$-linear set of rank $n$ in $\PG(k-1,q^n)$, $L=L(U)$ for some $\F_q$-subspace $U$ of rank $n$ of $V(k,q^n)$. The $\Fq$-subspace $U$ is isomorphic to $\Fn$ as an $\Fq$-vector space, let $\phi:U\mapsto \Fn$ be this isomorphism. Then $L(U)=\{\la \phi(x)\ra_{\Fn}:x\in \Fn^{\ast}\}$. Now $\phi(x)$ is an element of $V(k,q^n)$, and hence, can be written as $(f_1(x),\ldots,f_k(x))$ where $f_i=p_i\circ \phi$ and $p_i$ is the projection onto the $i$-th coordinate.

If there is a nonzero element in the intersection of the kernel of the $f_i$'s, then the rank of $L$ is strictly less than $n$. Finally if the maps $\{f_1,\ldots,f_k\}$ are linearly dependent over $\Fn$, then there exist $\alpha_i\in \Fn$, not all zero, such that $\sum_i \alpha_i f_i(x)=0$ for all $x$, implying that $L$ is contained in the hyperplane $(\alpha_1,\ldots,\alpha_k)^\perp$, and hence, $\dim\la L\ra \neq k-1$.

Vice versa, if all $f_i$ are $\Fq$-linear maps, then $\{(f_1(x),\ldots,f_k(x)):x\in \Fn\}$ defines an $\Fq$-subspace $U$ of rank $n$ if and only if $\ker(f_1)\cap\cdots\cap\ker(f_k)=\{0\}$. It follows that $L=L(U)$, and hence, $L$ is an $\Fq$-linear set.
 \end{proof}

We denote the $\Fq$-subspace of $V(k,q^n)$ used in this proof by $U_{f_1,\ldots,f_k}$, i.e.

  
\[U_{f_1,\ldots,f_k} := \{(f_1(x),\ldots,f_k(x)): x\in \Fn \},
\]

and the associated linear set
\[
L_{f_1,\ldots,f_k} := L(U_{f_1,\ldots,f_k})
\]
\begin{definition}
We say that two $\Fq$-subspaces of $V(k,q^n)$ are {\it equivalent} if there exists an element of $\GammaL(k,q^n)$ mapping $L_1$ to $L_2$. 

We say that two linear sets $L_1,L_2$ in $\PG(k-1,q^n)$ are {\it equivalent} if there exists an element of $\PGammaL(k,q^n)$ mapping $L_1$ to $L_2$. 
\end{definition}

\begin{remark}\label{slechtequiv}
Clearly, if the subspaces $U_1$ and $U_2$ are equivalent, then the linear sets $L(U_1)$ and $L(U_2)$ are equivalent. However the converse is not true; for example, defining
\begin{align*}
U_s&=\{(x,x^{q^s}):x\in \Fn\}\subset V(2,q^n),
\end{align*}
it holds that for any $s\notin \{1,n-1\}$ such that $(s,n)=1$, we have that $U_1$ is inequivalent to $U_s$, but $L(U_1)=L(U_s)$ (see
\cite[Remark 5.6]{CsajMarPol}).
\end{remark}

\begin{definition}\label{defwt}
For a linear set $L(U)$, we define the {\it weight} of a point $P$ defined by a vector $v\in V$ as 
\[
wt_{L(U)}(P):=\dim_{\Fq}(U\cap \langle v\rangle_{q^n}).
\]
For an $(s-1)$-dimensional subspace $\pi = \PG(W,q^n)$, where $W$ is an $s$-dimensional $\Fn$-subspace of $V$, we define, following \cite{olga}, the weight of $\pi$ with respect to $L(U)$ by
\[
wt_{L(U)}(\pi) := \dim_{\Fq}(U\cap W).
\]
\end{definition}

\begin{definition}
A linear set $L(U)$ is said to be {\it scattered} if the weight $w_{L(U)}(P)$ of any point $P$ is at most one.
\end{definition}

Scattered linear sets were introduced in \cite{BlLa2000}, and arise in various areas of finite geometry. See \cite{LavScat} for a recent survey on this topic.

\subsubsection{Linear sets as projected subgeometries}\label{proj}

We recall the following correspondence between linear sets and projected subgeometries from \cite{LP}. Let $\Sigma$ be a canonical subgeometry isomorphic to $\PG(s-1,q)$ of $\Sigma^*=\PG(s-1,q^n)$, let $\Lambda^*$ be a $(k-s-1)$- dimensional subspace of $\Sigma^*$ which is skew from $\Sigma$, and let $\Lambda$ be a $(k-1)$-dimensional subspace of $\Sigma^*$, skew from $\Lambda^*$, then the projection of $\Sigma$ from $\Lambda^*$ onto $\Lambda$, denoted by $p_{\Lambda^*,\Lambda}(\Sigma)$ defines an $\F_q$-linear set of rank $s$ in $\Lambda$. Vice versa, every $\F_q$-linear set of rank $s$ in $\PG(k-1,q^n)$ can be obtained in this way.

The following equivalent point of view for the weight of a point $P$ in a linear set, obtained as a projected subgeometry, has been used in the literature. However, by lack of an explicit proof for the equivalence of both definitions, give a proof here. This result will be used in Proposition \ref{BW}.

\begin{proposition} \label{weightproj} Let $L=p_{\Lambda^*,\Lambda}(\Sigma)$ be a linear set of rank $s$ in $\Lambda=\PG(k-1,q^n)$ obtained by projecting $\Sigma$, an $(s-1)$-dimensional $\Fq$-subgeometry of $\Sigma^*=\PG(s-1,q^n)$ from the $(s-k-1)$-dimensional subspace $\Lambda^*$ of $\Sigma^*$ onto $\Lambda$. Let $\Lambda^*=\PG(Y)$, where $Y$ is an $(s-k)$dimensional $\Fn$-vector space and let $\Lambda=\PG(Z)$, where $Z$ is a $k$-dimensional $\Fn$-vector space. Let $\Sigma=\PG(T)$, where $T$ is an $s$-dimensional $\Fq$-vectorspace, and let $\Sigma^*$ be $\PG(T^*)$ where $T^*=T\otimes\Fn$. Then $L=L(U)$ with $U=Z\cap (Y\oplus T)$, and if $\pi$ is a subspace of $\Lambda$, then $wt_{L(U)}(\pi)$, the weight of $\pi$ with respect to $L(U)$, is equal to one plus the dimension of the $\Fq$-subspace $(\la \pi,\Lambda^*\ra \cap \Sigma)$ of the subgeometry $\Sigma$.

\end{proposition}

\begin{proof}It follows from the construction, given in \cite[Theorem 1]{LP}, that  $p_{\Lambda^*,\Lambda}(\Sigma)=L(U)$ with $U=Z\cap (Y\oplus T)$, where $Y$ and $Z$ are considered as $\Fq$-subspaces of dimensions $n(s-k)$ and $nk$ respectively.

 Now by Definition \ref{defwt}, if $\pi=\PG(W)$, with $W\leq Z$, then $wt_{L(U)}(\pi)$ is equal to $\dim_{\Fq}(U\cap W)=\dim_{\Fq}((Z\cap (Y\oplus T))\cap W)$. Since $W\leq Z$, this is equal to $\dim_{\Fq}( (Y\oplus T)\cap W)$. Now  since $Y,T,W$ are disjoint $\Fq$-subspaces, this equals $\dim_{\Fq}(\la Y,T,W\ra)-\dim_{\Fq}(Y)-\dim_{\Fq}(T)-\dim_{\Fq}(W)=\dim_{\Fq}((Y\oplus W) \cap T)$. The subspace $(Y\oplus W) \cap T$ corresponds to the $\Fq$-subspace $(\la \pi,\Lambda^*\ra \cap \Sigma)$ of the subgeometry $\Sigma$ and the statement follows.
\end{proof}

\subsection{Linear maps and linearised polynomials}

\subsubsection{The spaces $\End_{\Fq}(\F_{q^n})$ and $\PG(\End_{\Fq}(\F_{q^n}))$}

Every $\Fq$-linear map from $\Fn$ to $\Fn$ can be uniquely represented as a linearised polynomial, i.e., as
$$f:x\mapsto \a_0x+\a_1x^q+\a_2x^{q^2}+\ldots+a_{n-1}x^{q^{n-1}},$$ where $\a_0,\ldots,a_{n-1}$ are elements of $\Fn$.
For $f,g$ $\Fq$-linear maps and $\a\in \Fn$, we have that $\a f:=x\mapsto \a f(x)$ and $f+g:=x\mapsto f(x)+g(x)$ are $\Fq$-linear maps as well. Hence, the set of $\Fq$-linear maps of $\Fn$ to $\Fn$, which is denoted by $\End_{\Fq}(\F_{q^n})$, forms an $\Fn$-vector space. Since there are $q^{n^2}$ such maps, $\End_{\Fq}(\F_{q^n})$ is $n$-dimensional over $\Fn$, and so $\E$ is isomorphic to $V(n,q^n)$. From now on we will write $V$ for $V(n,q^n)$.

We consider the projective space $\PG(\V)=\PG(n-1,q^n)$. Every point of $\PG(\V)$ is represented by an $\Fq$-linear map, defined up to a multiple of $\Fn$, so we denote the point corresponding to the map $f$ by $\la f\ran$.

Equivalently, we can make the correspondence $\V\mapsto \PG(n-1,q^n)$ explicit by defining the map $f\mapsto (\a_0,\a_1,\ldots,\a_{n-1})_{q^n}$, if $f:x\mapsto \a_0x+\a_1x^q+\a_2x^{q^2}+\ldots+a_{n-1}x^{q^{n-1}}$. Here, $(\a_0,\a_1,\ldots,\a_{n-1})_{q^n}$ denotes the projective point with homogeneous coordinates $(\a_0,\a_1,\ldots,\a_{n-1})$. We will abuse notation liberally throughout, by using the symbol $f$ to denote both a linearised polynomial and its coefficient vector, and switch freely between the two.

 \begin{definition} Given a linearised polynomial $f$, the {\it rank} of the corresponding vector in $V$ and point in $\PG(n-1,q^n)$ is the rank of $f$ as an element of $\E$. \end{definition}
\begin{definition} We will say a linearised polynomial (and its corresponding vector) is {\it invertible} if it has no non-zero roots in $\Fn$. \end{definition}

We denote the composition of two linearised polynomials $f,g$ as $f\circ g$, i.e. $(f\circ g)(x) := f(g(x))\mod x^{q^n}-x$. We define the dot product of two linearised polynomials $f=\sum \alpha_i x^i$, $g=\sum_i \beta_i x^i$ as the usual dot product of their coefficient vectors, i.e.
 \[
 f\cdot g := \sum_{i=0}^{n-1}\alpha_i \beta_i.
 \]

\subsubsection{The spaces $\Sigma_i$}

Every $\Fq$-linear map of rank $1$ is of the form $\a\Tr(\b x)$, where $\Tr$ is the trace map from $\Fn$ to $\Fq$, i.e., $\Tr:x\mapsto x+x^q+\ldots+x^{q^{n-1}}$. As $\la\a\Tr(\b x)\ran=\la \a' \Tr(\b x)\ran$ and $\la\Tr(\b x)\ran=\la \Tr(\lambda \b x)\ran$ if $\lambda \in \Fq^\ast$, we see that there are exactly $\frac{q^n-1}{q-1}$ projective points  $\la f\ran$ with $f$ a map of rank $1$. Call this set $\Sigma$. We get that
$$\Sigma = \{(\b,\b^q,\b^{q^2},\ldots,\b^{q^{n-1}})_{q^n}\mid \b\in \Fn^\ast\}$$ and
we see that $\Sigma$ defines an $\Fq$-subgeometry of $\PG(n-1,q^n)$.

In general, we define subsets $\Sigma_i$ of $\PG(n-1,q^n)$ by
\[
\Sigma_i = \{(\a_0,\a_1,\ldots,\a_{n-1})_{q^n} : \rank(\a_0x+\a_1x^q+\a_2x^{q^2}+\ldots+a_{n-1}x^{q^{n-1}}) \leq i\}.
\]
or, equivalently,
\[
\Sigma_i = \{\langle f\rangle_{q^n}: f\in V,\rank(f) \leq i\}.
\]
We see that $\Sigma_1$ equals the subgeometry $\Sigma$. It is well-known that any rank $k$ map can be written as the sum of $k$ rank $1$ maps. This means that the points of $\PG(\End_{\Fq}(\F_{q^n}))$ corresponding to maps of rank $k$ are all of the form $$\la \psi_1\Tr(\a_1 x)+\psi_2\Tr(\a_2 x)+\ldots+\psi_k\Tr(\a_k x)\ran$$ for some $\psi_1,\psi_2,\ldots,\psi_k$, and $\a_1,\ldots, \a_k$ in $\Fn$. Geometrically, these points are the points that lie on a subspace spanned by $k$ points of $\Omega$. This means that $\Sigma_i$ is the $(i-1)$-st {\em secant variety} of $\Sigma$. We see that the points in $\Sigma_{i}\backslash \Sigma_{i-1}$ are precisely the points of rank $i$.

\subsection{Rank-metric codes}
\subsubsection{Definition}
A {\it rank-metric code} is a set of maps $\C\subset \End_{\Fq}(\Fn)\simeq V$, with distance defined by the rank-distance;
\[
d(f,g) = \rank(f-g).
\]
As outlined above, we may regard $\C$ as a set of linearised polynomials. We define (following the notation of \cite{Lunardon}) the set $\Omega(\C)\subset \PG(n-1,q^n)$ by
\[
\Omega(\C) := \{\langle f\rangle_{q^n}:f\in \C^\ast\}.
\]
If $\C$ is an $\Fq$-subspace of $V$, then $\Omega(\C)$ is a linear set in $\PG(n-1,q^n)$. If $\C$ is an $\Fn$-subspace of $V$, then $\Omega(\C)$ is a subspace of $\PG(n-1,q^n)$.

\begin{definition}
A set $\C\subset V$ is called a {\it maximum rank distance} code if $|\C|=q^{nk}$ and the rank of any nonzero $f\in \C$ is at least $n-k+1$. The following is immediate by definition (see also \cite{Lunardon}).
\end{definition}

\begin{proposition}
A subset $\C$ of $V$ of size $q^{nk}$ is a maximum rank distance code if and only if $\langle f-g\rangle_{\Fn} \notin \Sigma_{n-k}$ for all $f,g\in \C$ with $f\ne g$.

An additively closed subset $\C$ of $V$ of size $q^{nk}$ is a maximum rank distance code if and only if $\Omega(\C)$ is disjoint from $\Sigma_{n-k}$.
\end{proposition}

This setup is very similar as for the geometric construction of (spread sets for) {\em semifields}, proven in \cite{LavSem}.

\subsubsection{Equivalence for rank-metric codes}

\begin{definition}
Two rank-metric codes $\C_1$ and $\C_2$ are said to be {\it equivalent} if there exist linearised polynomials $g,h,k\in \E$, with $g,h$ invertible, and an automorphism $\rho$ of $\Fn$ (not necessarily fixing $\Fq$), such that
\[
\C_2 = \{g\circ f^\rho\circ h +k : f\in \C_1\}.
\]

The action of $\rho \in \Aut(\F_{q^n})$ is defined by $f^\rho := \sum_{i=0}^{n-1} \a_i^\rho x^{q^i}$; i.e. $f^\rho(x) = (f(x^{\rho^{-1}}))^\rho$.
\end{definition}
 If $\C_1,\C_2$ are $\Fq$-subspaces of $\E$, then we may assume $k=0$. 

\begin{definition}(see \cite{morrison})
Two rank-metric codes $\C_1$ and $\C_2$ are said to be {\it semilinearly equivalent} if there exist an invertible linearised polynomial $h\in \E$, an integer $m\in \{0,\ldots,n-1\}$, and an automorphism $\rho$ of $\F_{q^n}$, such that
\[
\C_2 = \{x^{q^m}\circ f^\rho \circ h : f\in \C_1\}.
\]
\end{definition}
Clearly semilinear equivalence implies equivalence. 

%



\begin{proposition}
\label{prop:Fnpreserve}
Suppose $\C$ is an $\Fn$-subspace of $V$. Suppose $g(x) = \sum_{m=0}^{n-1}g_mx^{q^m}$ is an invertible linearised polynomial such that $g\circ \C$ is also an $\Fn$-subspace of $V$. Then for every $m$ such that $g_m\ne 0$, it holds that $g\circ \C = x^{q^m}\circ \C$.
\end{proposition}

\begin{proof}
Suppose $\C$ is $k$-dimensional over $\Fn$. Then $g\circ \C$ is an $\Fn$-subspace if and only if there exists an $(n-k)$-dimensional subspace $\H$ of $V$ such that $(g\circ f)\cdot h=0$ for all $f\in \C,h\in \H$. Furthermore $g'\in g\circ \C$ if and only if $g'\cdot h=0$ for all $h\in \H$. Put $f=\sum_{i=0}^{n-1}f_ix^{q^i}$, $g=\sum_{m=0}^{n-1}g_mx^{q^m}$ and $h=\sum_{i=0}^{n-1}h_i x^{q^i}$. Now as $\alpha f \in \C$ for all $f\in \C$ and all $\alpha \in \Fn$, then for any $h\in \H$ we have
\begin{align*}
(g\circ \alpha f)\cdot h &= \left(\sum_{i=0}^{n-1} \left(\sum_{m=0}^{n-1} g_m \alpha^{q^m} f_{i-m}^{q^m}\right)x^{q^i}\right)\cdot \left(\sum_{i=0}^{n-1} h_i x^{q^i}\right),
\end{align*}
where we have used the convention that $f_l=f_{n-l}$ if $l<0$. Now we find that 
\begin{align*}
\left(\sum_{i=0}^{n-1} \left(\sum_{m=0}^{n-1} g_m \alpha^{q^m} f_{i-m}^{q^m}\right)x^{q^i}\right)\cdot \left(\sum_{i=0}^{n-1} h_i x^{q^i}\right)&=  \sum_{i=0}^{n-1} \left(\sum_{m=0}^{n-1} g_m \alpha^{q^m} f_{i-m}^{q^m}\right)h_i\\
&= \sum_{m=0}^{n-1} g_m \left(\sum_{i=0}^{n-1} f_{i-m}^{q^{\color{blue}m}} h_i\right) \alpha^{q^m}\\
&= \sum_{m=0}^{n-1} g_m ((x^{q^m}\circ f) \cdot h)\alpha^{q^m}\\
&=0
\end{align*}
for all $\alpha \in \Fn$. Thus for each $m$ we must have $g_m ((x^{q^m}\circ f) \cdot h)=0$, and so if $g_m\ne 0$ then $(x^{q^m}\circ f) \cdot h = 0$ for all $f\in \C,h\in \H$. But then $x^{q^m}\circ f\in g\circ \C$ for all $f\in \C$, implying $g\circ \C = x^{q^m}\circ \C$, as claimed.
\end{proof}

\begin{proposition}
\label{prop:Fnequiv}
Suppose $\C_1$ and $\C_2$ are $\Fn$-subspaces of $V$. Then $\C_1$ and $\C_2$ are equivalent if and only if they are semilinearly equivalent.
\end{proposition}

\begin{proof}
If $\C_1$ and $\C_2$ are semilinearly equivalent, then by definition they are equivalent. Suppose now $\C_1$ and $\C_2$ are equivalent, i.e. $\C_2^\rho = g\circ \C_1 \circ h$ for some invertible $g,h$. Now $\C_2^\rho\circ h^{-1}$ is an $\Fn$-subspace, and hence, $g\circ \C_1$ is an $\Fn$-subspace. So, by Proposition \ref{prop:Fnpreserve} there exists an $m$ such that $g\circ \C_1 = x^{q^m}\circ \C_1$. Therefore $\C_2^\rho = x^{q^m} \circ \C_1 \circ h$, and so $\C_1$ and $\C_2$ are semilinearly-equivalent, proving the claim.
\end{proof}

\begin{remark}
If we regard $V$ as a vector space over $\Fq$, then the set of rank one maps defines a {\it Segre variety} in $\PG(V,\Fq)\simeq \PG(n^2-1,q)$. The set of one-dimensional $\Fn$-subspaces of $V$ corresponds to a desarguesian spread $\D$, and the Segre variety is partitioned by these spaces. This is the field-reduction of the subgeometry $\PG(n-1,q)$. $\Fn$-subspaces are then precisely those subspaces spanned by elements of $\D$

The collineation group of $\PG(n^2-1,q)$ is $\PGammaL(n^2,q)$. The subgroup of this fixing $\D$ is isomorphic to $\PGammaL(n,q^n)$.

The group induced by the set of equivalences is the set of elements of $\PGammaL(n^2,q)$ which fix the field-reduced subgeometry. 

The group induced by the set of semilinear equivalences is the set of elements of $\PGammaL(n,q^n)$ which fix the field-reduced subgeometry. It is well-known that this has the form $\Aut(\Fn:\Fq).\PGammaL(n,q)$.

%
%
%
%
%

Hence what we have shown here is that two subspaces of $\PG(n^2-1,q)$ obtained by field-reduction of subspaces of $\PG(n-1,q^n)$ are equivalent under the stabiliser in $\PGammaL(n^2,q)$ of the field-reduced subgeometry if and only if the two subspaces are equivalent under the stabiliser in $\PGammaL(n,q^n)$ of the field-reduced subgeometry.


%
We suspect that this result may be known. However as we could not find an exact reference, we chose to include a proof. Similar ideas can be found in for example \cite{LiebholdNebe}, \cite{morrison}, though neither imply this result.
\end{remark}




\subsection{Duality in $\E$ and in $\PG(n-1,q^n)$}

Delsarte considered duality in $V=\E$ by representing endomorphisms as matrices over $\Fq$ and defining a symmetric bilinear form
\[
(A,B)\mapsto \Tr(AB^T),
\]
where $\Tr$ denotes the matrix trace. He showed that using this inner product, the rank-distribution of a subspace and its dual are related by a rank-metric version of the MacWilliams identities \cite{Delsarte}. 

In this paper we will use a different symmetric bilinear form, more suited to working with linearised polynomials, following \cite{Sheekey}. The dual of a subspace with respect to this form is equivalent to the dual with respect to the form used by Delsarte. 

For two elements of $V$ given by $f:x\mapsto \sum_{i=0}^{n-1}f_i x^{q^i}$, $g:x\mapsto \sum_{i=0}^{n-1}g_i x^{q^i}$, we define the symmetric bilinear form
\[
(f,g)\mapsto \tr\left(\sum_{i=0}^{n-1} f_i g_i\right),
\]
where $\tr$ denotes the field trace from $\Fn$ to $\Fq$. 

The {\em (Delsarte) dual} of an $\F_q$-subspace $\C$ of $V$ is then defined as 
$$\C^{\bot}=\left\{g \in V : \tr\left(\sum_{i=0}^{n-1} f_i g_i\right)=0\ \forall f\in \C\right\}.$$

If $\C$ is an $\Fn$-subspace of $V$, then it is easy to check that the dual of $\C$ with respect to this form is equal to the dual of $\C$ with respect to the form
\[
(f,g)\mapsto \sum_{i=0}^{n-1} f_i g_i = f\cdot g.
\]

Thus we may alternatively define the {\em (Delsarte) dual} of an $\Fn$-subspace $\C$ of $\E$ as 
$$\C^{\bot}=\{g \in V : f \cdot g=0\ \forall f\in \C\}.$$
\begin{remark}
Care should be taken when considering duality for rank-metric codes in the non-square case, as duals with respect to this dot product may not be equivalent to duals with respect to Delsarte's form. In this paper we are only concerned with the square case.
\end{remark}

The (Delsarte) dual of an MRD code is again an MRD code. For more information about duality of rank-metric codes, we refer to \cite{Ravagnani}.

The Delsarte dual operation $\bot$ on $\E$ induces a dual operation on $\PG(n-1,q^n)$, which we also will denote by $\bot$. In particular, as seen in the previous subsection, an $\F_{q^n}$-linear subspace $\C$ of $\E$ corresponds to a projective subspace $\Omega(\C)$ of $\PG(n-1,q^n)$ and we have that 

 $$\Omega(\C^\perp)=\Omega(\C)^\perp.$$

\section{Linear sets from MRD codes}

In this section we provide a geometric interpretation of the correspondence outlined in \cite{Sheekey} between scattered linear sets on a projective line and certain classes of MRD codes. In Section \ref{new} we will extend this correspondence to MRD codes of higher dimensions. We will incorporate the notion of Delsarte duality of a rank-metric code into this geometric picture, which will allow us to exploit the MacWilliams identities for rank-metric codes to investigate linear sets in Section \ref{mac}. 

\subsection{Linear sets in $\PG(k-1,q^n)$ from $\Fn$-subspaces of $V$}\label{omegas}

Suppose $\C$ is a $k$-dimensional $\Fn$-subspace of $V$, and hence $\Omega(\C)$ a $(k-1)$-dimensional subspace of $\PG(n-1,q^n)$. Recall that $\Sigma$ is an $\Fq$-subgeometry of $\PG(n-1,q^n)$.

\begin{proposition}\label{omegabot}
If $\C$ is a $k$-dimensional $\Fn$-subspace of $V$, then $\{x\in \Fn:g(x)=0~\forall g\in \C\}=\{0\}$ if and only if $\Omega(\C)^{\perp}\cap \Sigma=\emptyset$.
\end{proposition}

\begin{proof} If $g\in \C$, then $g$ is a linearised polynomial, so we can write $g(x)=\sum_{i=0}^{n-1} g_i x^{q^i}$.
A point $(x,x^q,\ldots,x^{q^{n-1}})_{q^n}\in \Sigma$ is in $\Omega(\C)^\perp$ if and only if $(x,x^q,\ldots,x^{q^{n-1}})\cdot (g_0,\ldots,g_{n-1})=0$ for all $g\in \C$, if and only if $\sum_{i=0}^{n-1} g_i x^{q^i}=0$ for all $g\in \C$, if and only if $g(x)=0$ for all $g\in \C$.
\end{proof}



As we saw in Lemma \ref{linsetpol}, a $k$-tuple of linearised polynomials gives an $\Fq$-subspace of $V(k,q^n)$
$$U_{f_1,\ldots,f_k}=\{(f_1(x),\ldots,f_k(x))\vert x\in \mathbb{F}_{q^n}\},$$
and a linear set in $\PG(k-1,q^n)$
$$L_{f_1,\ldots,f_k}:=\{(f_1(x),\ldots,f_k(x))_{q^n}\vert x\in \mathbb{F}_{q^n}^*\} = L(U_{f_1,\ldots,f_k}).$$

Thus from a $k$-dimensional $\Fn$-subspace $\C$ of $V$, we can define a family of linear sets in $\PG(k-1,q^n)$.
\begin{definition}
For a $k$-dimensional $\Fn$-subspace $\C$ of $V$, we define a set of $\Fq$-subspaces of $V(k,q^n)$ by
$$\U_\C:=\{U_{f_1,\ldots,f_k}:f_i \in V, \C = \langle f_1,\ldots,f_k\rangle_{q^n}\},$$
and a set of linear sets in $\PG(k-1,q^n)$ by
$$\L_\C:=\{L_{f_1,\ldots,f_k}:f_i \in V, \C = \langle f_1,\ldots,f_k\rangle_{q^n}\}.$$
\end{definition}


Different choices of basis for $\C$ may give different linear sets. However, it is easy to prove that they are all $\PGL(k,q^n)$-equivalent.

\begin{proposition}\label{propeq}
Suppose $\C$ is a $k$-dimensional $\Fn$-subspace of $V$ such that $\Omega(\C)^{\perp}\cap \Sigma=\emptyset$. Then any two elements of $\L_\C$ are $\PGL(k,q^n)$-equivalent linear sets of rank $n$ in $\PG(k-1,q^n)$. Conversely, for any $L\in \L_\C$ and any $\phi \in \PGL(k,q^n)$, it holds that $L^\phi \in \L_\C$. Thus $\L_\C$ is a $\PGL(k,q^n)$-equivalence class of linear sets of rank $n$ in $\PG(k-1,q^n)$.

%

%
%
%
\end{proposition}

\begin{proof}
Let $L_{f_1,\ldots,f_k}$ be an element of $\L_\C$. If there would be a non-zero element $y\in \Fn$ such that $f_1(y)=\ldots=f_k(y)=0$, then, as $f_1,\ldots,f_k$ is a basis for $\C$, $y$ would be contained in $\{x\in \Fn:g(x)=0~\forall g\in \C\}$, a contradiction by Lemma \ref{omegabot}, since $\Omega(\C)^{\perp}\cap \Sigma=\emptyset$ by our assumption. So this implies that $f_1,\ldots,f_k$ have no non-trivial common zeroes and that $L_{f_1,\ldots,f_k}$ has rank $n$ by Lemma \ref{linsetpol}.


Now consider $L_{f_1,\ldots,f_k}$ and $L_{f_1',\ldots,f_k'}$ in $\L_\C$. As $f_1,\ldots,f_k$ and $f_1',\ldots,f_k'$ each form a basis for $\C$, we know that there is an element $\phi$ of $\GL(k,q^n)$ such that $f_j' = \sum_{i=1}^k \phi_{ij}f_i$. Then 
\[
(f_1(x),\ldots,f_k(x))\phi = (f_1'(x),\ldots,f_k'(x))
\]
for all $x$, and so $\la \phi\ra_{q^n}$ is an element of $\PGL(k,q^n)$ mapping $L_{f_1,\ldots,f_k}$ onto $L_{f_1',\ldots,f_k'}$.

Conversely, for any $\phi$ of $\GL(k,q^n)$, and $L_{f_1,\ldots,f_k}\in \L_\C$, define $f_j' = \sum_{i=1}^k \phi_{ij}f_i$. Then $f_1',\ldots,f_k'$ forms a basis for $\C$, and
\[
(L_{f_1,\ldots,f_k})^{\la \phi\ra_{q^n}} = L_{f_1',\ldots,f_k'}\in \L_\C.
\]
\end{proof}

\begin{proposition}\label{propeq2}
If $\C$ and $\C'$ are equivalent $\Fn$-subspaces of $V$, then there exists for every $L\in \L_\C$, an element $\phi\in \PGammaL(k,q^n)$ such that $L^\phi\in \L_{\C'}$.
A linear set $L$ is $\PGammaL(k,q^n)$-equivalent to an element of $\L_\C$ if and only if $L\in \L_{\C^\sigma}$ for some $\sigma\in \Aut(\Fn)$. Thus 
\[
\bigcup_{\sigma\in \Aut(\Fn)} \L_{\C^\sigma}
\]
is a $\PGammaL(k,q^n)$-equivalence class of linear sets of rank $n$ in $\PG(k-1,q^n)$.
\end{proposition}

\begin{proof}
Suppose $\C$ and $\C'$ are equivalent $\Fn$-subspaces of $V$. By Proposition \ref{prop:Fnequiv}, we have that $\C' = x^{q^m}\circ \C^{\rho}\circ h$ for some invertible $h$, some integer $m$, and some automorphism $\rho$ of $\F_{q^n}$. For any basis $f_1,\ldots,f_k$ of $\C$, it holds that $x^{q^m}\circ f_1^{\rho}\circ h,\ldots, x^{q^m}\circ f_k^{\rho}\circ h$ is a basis for $\C'$.

Thus
\[
\{((x^{q^m}\circ f_1^{\rho}\circ h)(x),\ldots,(x^{q^m}\circ f_k^{\rho}\circ h)(x))_{q^n}\vert x\in \mathbb{F}_{q^n}^*\} \in \L_{\C'}.
\]

%
%

 Defining $y=h(x)^{\rho^{-1}}$, gives
\[\{(f_1(y)^{q^m\rho},\ldots,f_k(y)^{q^m\rho})_{q^n}\vert y\in \mathbb{F}_{q^n}^*\} \in \L_{\C'}.\]

This is clearly $\PGammaL(k,q^n)$-equivalent to $L_{f_1,\ldots,f_k}\in \L_\C$.
\end{proof}
We will give a geometric reason for Proposition \ref{propeq2} in the next section. 

\begin{remark}
Note that these results do not imply that if $\L_\C$ and $\L_{\C'}$ are $\PGammaL(k,q^n)$-equivalent then $\C$ and $\C'$ are equivalent; this is not always true. See Remark \ref{slechtequiv}, \ref{mis} and \cite{CsajZan}, \cite{CsajMarPol} for examples in the case $k=2$. 
\end{remark}

\subsection{A geometric interpretation}
\label{clou}
In Subsection \ref{omegas}, we have, for a subspace $\C$ of $V$, defined a set of linear sets $\L_\C$. Starting from  the subspace $\C$ of $V$, we will now define a different linear set, as in \cite{Lunardon}. As explained in Subsection \ref{proj}, linear sets can be constructed as projected subgeometries. Using the subgeometry $\Sigma$ of $\PG(V)$, corresponding to the linear maps of rank one, and an $\Fn-$subspace $\C$ of $V$, then
\[
\Sigma/\Omega(\C) \subseteq \PG(n-k-1,q^n).
\] 
defines a linear set $L'$ of rank $n$ in $\PG(n-k-1,q^n)$.

Note that in general $L'$ is not a linear set in $\L_\C$, as $\L_\C$ is a set of linear sets in $\PG(k-1,q^n)$ whereas $L'$ lies in $\PG(n-k-1,q^n)$. However we will now show that there is a geometric connection between the two constructions, our main result of this section.
\begin{theorem}
\label{thm:main}
Let $\C$ be an $\Fn$-subspace of $V$ with $\Omega(\C)^\perp\cap \Sigma=\emptyset$. Then for any $\Fn$-basis $\{ f_1,\ldots,f_k\}$ of $\C$ we have
$$U_{f_1,\ldots,f_k}\cong V_0/\C^\perp,$$
where $V_0$ is an $\Fq$-subspace of $V$ satisfying $\Omega(V_0) = \Sigma$, and 
$$L_{f_1,\ldots,f_k}\cong \Sigma/\Omega(\C)^\perp.$$
Therefore
\[
\Sigma/\Omega(\C)^\perp \in \L_\C.
\]
%
\end{theorem}

\begin{proof} 

Pick a $(k-1)$-space, say $\Omega(\C_1)=\la  g_1,\ldots,g_{k}\ra_{q^n}$ skew from $\Omega(\C)^\perp$. Then the quotient space $\Sigma/\Omega(\C)^\perp$ is isomorphic to the intersection of the space $\Omega(\C_1)$ with all spaces of the form $\la x,\Omega(\C)^\perp\ra$, where $x\in \Sigma$. We conclude that $\Sigma/\Omega(\C)^\perp$ is isomorphic to the set $M$ of points of the form $\la a_1 g_1+\cdots+a_{k}g_{k}\ra_{q^n}$, with $a_i \in \Fn$ such that $a_1 g_1+\cdots+a_{k}g_{k} = \bar{\beta}+t$ for some $\la\bar{\beta}\ra_{q^n}\in \Sigma,\la t\ra_{q^n}\in \Omega(\C)^\perp$.

Since $\la g_1\ra_{q^n},\ldots,\la g_k\ra_{q^n}$ are different points of $\PG(n-1,q^n)$, the set $M$ contained in $\Omega(\C_1)$ is clearly isomorphic to $M'$ corresponding to the points defined by the vectors from the set 
\[
\left\{(a_1,\ldots,a_k):\sum_{i=1}^k a_i g_i =\bar{\beta}+t, \la\bar{\beta}\ra_{q^n}\in \Sigma,\la t\ra_{q^n}\in \Omega(\C)^\perp\right\}.
\]

Now suppose that 
\begin{align}\sum_{i=1}^k a_i g_i=\bar{\beta}+t
\end{align}
for $a_i\in \F_{q^n}$ and $\la\bar{\beta}\ra_{q^n}\in \Sigma,\la t\ra_{q^n}\in \Omega(\C)^\perp$. Take the scalar product with $f_j$ on both sides of (1) to find
$$\sum_{i=1}^k a_i (f_j\cdot g_i)=f_j\cdot\bar{\beta}+f_j \cdot t$$
Now recall that $f_j\cdot \bar{\beta}=f_j(\beta)$, and that $f_j\cdot t=0$ since $\la t\ra_{q^n}\in \Omega(\C)^\perp$.
Hence we get that
$$\sum_{i=1}^k a_i (f_j\cdot g_i)=f_j(\beta)$$
for all $j$, which we rewrite as
$$\begin{pmatrix}a_1&\cdots&a_k\end{pmatrix}\phi=\begin{pmatrix}f_1(\beta)&\cdots&f_k(\beta)\end{pmatrix},$$

where $\phi$ is the $k\times k$ matrix with $\phi_{ij} = f_j\cdot g_i$.

We show that $\phi$ is non-singular. 
Suppose that this matrix is singular. Then we can find a non-trivial linear combination of its columns that gives the zero row, i.e. we can find  some $\alpha_1,\ldots,\alpha_k \in \F_{q^n}$, not all zero, such that
\[
\sum_{i=1}^k \alpha_i (g_i\cdot f_j) =0
\]
for all $j$. Rearranging this, we find that $f_j \cdot (\sum_{i=1}^k \alpha_i g_i )=0$ for all $j$, implying $\sum_{i=1}^k \alpha_i g_i \in \Omega(\C)^\perp \cap \Omega(\C_1)$, a contradiction since $\Omega(\C)^\perp$ and $\Omega(\C_1)$ are disjoint. Hence $\phi$ is non-singular, as claimed. 

Now applying $\phi^{-1}$ to $M'$, we get that $M'$ is isomorphic to the set of points defined by vectors in the set
\[
\{\begin{pmatrix}f_1(\beta)&\cdots&f_k(\beta)\end{pmatrix} : \beta \in \Fn^\ast\},
\]
which by definition is equal to $L_{f_1,\ldots,f_k}$.

We conclude that $\Sigma/\Omega(\C)^\perp\cong M\cong M'\cong L_{f_1,\ldots,f_k} $, as claimed. As we have only used $\PGL$-equivalences, using Proposition \ref{propeq}, we find that $\Sigma/\Omega(\C)^\perp\in \L_\C$.
\end{proof}

We can use this result to give a geometric proof of part of Proposition \ref{propeq2}.

\begin{proof}[Alternate proof of Proposition \ref{propeq2}] Let $\C_1$ and $\C_2$ be equivalent subspaces. Let $L_1=\Sigma/\Omega(\C_i)^\perp$.
By Theorem \ref{thm:main}, we have that $\Sigma/\Omega(\C_i)^\perp\in \L_{\C_i}$ for $i =1,2$. By Proposition \ref{prop:Fnequiv} $\C_1$ and $\C_2$, are equivalent if and only if they are semilinearly equivalent. It is straightforward to check that $\C_1$ and $\C_2$ are semilinearly equivalent if and only if $\C_1^\perp$ and $\C_2^\perp$ are semilinearly equivalent, in which case there is a collineation $\phi$ of $\PG(n-1,q^n)$ fixing $\Sigma$ and mapping $\Omega(\C_1)^\perp$ to $\Omega(\C_2)^\perp$.


Then clearly
\[
L_1^{\phi}= (\Sigma/\Omega(\C_1)^\perp)^\phi = \Sigma^\phi /(\Omega(\C_1)^\perp)^\phi = \Sigma/\Omega(\C_2)^\perp\in \L_{\C_2},
\]
as claimed. \end{proof}

%
%
%

\begin{remark}
We note again that the converse is not necessarily true. Counterexamples can occur when the cone defined by $\Omega(\C_1)$ and $\Sigma$ contains other subgeometries. Counterexamples are known only in the case $k=2$. However we do have the following generalisation of \cite[Theorem 8]{Sheekey}.
\end{remark}

\begin{proposition}\label{equi}

Suppose $\C_1,\C_2$ are $\Fn$-subspaces of $V\simeq \E$. Then $U_{\C_1}$ and $U_{\C_2}$ are $\GammaL(k,q^n)$-equivalent if and only if $\C_1$ and $\C_2$ are equivalent.
\end{proposition}


We conclude this section with a definition.

\begin{definition} The {\it companion} of the set $\U_\C$ of $n$-dimensional $\Fq$-subspaces of  $V(k,q^n)$ is the set $\U_{\C^\perp}$ of  $n$-dimensional $\Fq$-subspaces of  $V(n-k,q^n)$.
\end{definition}

\begin{remark}\label{mis} The corresponding statement for linear sets does not appear to be well-defined. If $\C$ and $\C'$ are inequivalent subspaces but define $\PGammaL$-equivalent linear sets, their duals do not necessarily define $\PGammaL$-equivalent linear sets.
Let $\C_1 = \pt{x,x^q}$ and $\C_2=\pt{x,x^{q^2}}$ be subspaces in $\PG(4,q^5)$, then $\C_1$ and $\C_2$ are inequivalent, but $L_{x,x^q}=L_{x,x^{q^2}}$, and hence, $L_{x,x^{q^2}} \in \L_{\C_1}$ (see \cite{CsajZan}). We have $\C_1^\perp = \pt{x^{q^2},x^{q^3},x^{q^4}}$ and $\C_2^\perp=\pt{x^{q},x^{q^3},x^{q^4}}$. It is not hard too check (e.g. using the GAP package FinInG \cite{fining}) that $L_{x^{q^2},x^{q^3},x^{q^4}}$ and $L_{x^{q},x^{q^3},x^{q^4}}$ are not $\PGammaL$-equivalent for $q=2,3,4$, and hence, $L_{x^{q},x^{q^3},x^{q^4}}\notin \L_{{\C_1}^\perp}$.


However, we will show in the next section that there is a correspondence between the weight distribution of points and hyperplanes in the linear sets in $\L_\C$ and $\L_{\C^\perp}$. So, even though one can take a linear set in $\L_{\C_1^\perp}$ and one in $\L_{\C_2^\perp}$ which are not equivalent, if $\C_1$ and $\C_2$ are equivalent, these will have the same weight distribution.

\end{remark}

\section{The extension of Sheekey's connection between scattered linear sets and MRD codes}\label{new}

Recall that if $\C\subset V$ is an $\F_{q^n}$-linear rank metric code, $\Omega(\C)$ is a subspace of $\PG(n-1,q^n)$. We now relate the rank distribution of the subspace $\Omega(\C)$ in $\PG(n-1,q^n)$ to properties of a linear set in $\L_\C$.

\begin{definition}
The {\it rank distribution} of a subset $\Omega(\C)$ is defined as the vector that has the number of points of $\Omega(\C)$ with rank $i$ on the $i$-th position, i.e., 
\[
rk(\Omega(\C)) := (v_1,\ldots,v_n),\ \mathrm{ with\ } v_i=\mathrm{ number\ of\ points\ of\ }\Omega(\C)\ \mathrm{ with\ rank\ }i.
\]
\end{definition}
\begin{definition}

The {\it weight distribution with respect to $s$-subspaces} of an $\Fq$-linear set $L(U)$ in $\PG(k-1,q^n)$ is defined as the vector
\[
w_s(L(U)) := (v_1,\ldots,v_n),\ \mathrm{ with\ } v_i=\mathrm{ number\ of\ } s-\mathrm{ spaces\ of\ }\PG(k-1,q^n)\ \mathrm{ with\ }wt_{L(U)}(\pi)=i.
\]

We call $w_0(L)$ the {\it weight distribution of $L(U)$}, and we call $w_{k-2}(L(U))$ the {\it weight distribution of $L$ with respect to hyperplanes}.
\end{definition}

We see that $L(U)$ is scattered if $w_0(L)$ has all its entries in $\{0,1\}$.


In \cite{Sheekey}, the following was shown (adapted to the notation of this paper).
\begin{theorem}
Let $\C$ be a two-dimensional $\Fn$-subspace of $V$ (or, equivalently, $\Omega(\C)$ a line in $\PG(n-1,q^n)$. Then a linear set $L\in L_\C$ is scattered if and only if $\C$ is disjoint from $\Sigma_{n-2}$, i.e. if and only if $\C$ is a two-dimensional $\Fn$-linear maximum rank-distance code.
\end{theorem}

We now aim to extend this correspondence to subspaces $\Omega(\C)$ of larger dimension. However it is not true in general to say that if a linear set $L\in L_\C$ is scattered, then $\C$ is maximum rank-distance, as we will illustrate in a later example. In order to characterise linear sets arising from $\Fn$-linear MRD codes, we introduce a new concept for linear sets; that of being {\it scattered with respect to hyperplanes}.

\begin{definition}
An $\F_q$-linear set $L(U)$ of rank $n$ in $\PG(k-1,q^n)$, with $\dim\la L(U)\ra =k-1$,
 is said to be {\it scattered with respect to hyperplanes} if the rank of $L(U)\cap H$ is at most $k-1$ for all hyperplanes $H$ of $\PG(k-1,q^n)$.
\end{definition} 
 It is easy to see that an $\F_q$-linear set which is scattered with respect to hyperplanes is necessarily a scattered linear set. Moreover, for every linear set $L$ of rank $n$ in $\PG(k-1,q^n)$, it is easy to find some hyperplane meeting $L$ in a linear set of rank at least $k-1$. Hence, $k-1$ is the smallest possible upper bound for the rank of the intersection of a linear set of rank $n$ with a hyperplane of $\PG(k-1,q)$.

\begin{proposition}\label{weightrank}
Let $L_{f_1,\ldots,f_k} \in L_\C$ be a linear set in $\PG(k-1,q^n)$ defined by a subspace $\Omega(\C) = \langle f_1,\ldots,f_k\rangle_{q^n}\leq \PG(n-1,q^n)$ of dimension $k-1$. Then for any hyperplane $H = (a_1,\ldots,a_k)^{\perp}$, we have that
\[
wt_{L_{f_1,\ldots,f_k}}(H)  = n - \rank\left(\sum_{i=1}^k a_i f_i\right).
\]
\end{proposition}

\begin{proof}
Let $P = \langle(f_1(x_0),\ldots,f_k(x_0))\rangle_{q^n}$, with $x_0\neq 0 \in \Fn$ be a point of $L_{f_1,\ldots,f_k}$. Then $P\in H$ if and only if $a_1f_1(x_0)+\cdots+a_kf_k(x_0)=0$ 
, which occurs if and only if $x_0\in \ker(a_1f_1+\cdots a_kf_k)$. Hence $L_{f_1,\ldots,f_k}$ defines a linear set such that the weight of the hyperplane $H$ with respect to it is rank $\dim(\ker(a_1f_1+\cdots a_kf_k)) = n-\rank\left(\sum_{i=1}^k a_i f_i\right)$, as claimed.
\end{proof}

We get from the previous proposition that $w_{k-2}(L_{f_1,f_2,\ldots,f_k})=n.{\bf 1}-rk(\Omega(\C))$. As the right hand side is independent of the basis $f_1,\ldots,f_k$ we have chosen for $\C$, we find:

\begin{corollary}
The weight distribution of a linear set $L\in \L_\C$ with respect to hyperplanes is determined by the rank distribution of points in $\Omega(\C)$:
\[
w_{k-2}(L)= n.{\bf 1}-rk(\Omega(\C)),
\]
where $\bf{1}$ is the all-one vector.
\end{corollary}


We get from the previous corollary that the linear set $L$ is scattered with respect to hyperplanes if and only if all the points of $\Omega(\C)$ have rank at least $n-k+1$, which leads to the following statement. Note that when $k=2$, points are hyperplanes, and so setting $k=2$ returns the construction of the MRD codes of dimension 2 given in \cite[Section 5 ]{Sheekey}.

\begin{corollary}
A linear set $L\in \L_\C$ with $\dim \langle L\rangle=k-1$ is scattered with respect to hyperplanes if and only if $\Omega(\C)$ is disjoint from $\Sigma_{n-k}$ in $\PG(n-1,q^n)$, if and only if $\C$ is an $\Fn$-linear MRD code with minimum distance $n-k+1$.
\end{corollary}

\begin{remark}
An example for which a linear set $L(\C)$ is scattered, but not scattered with respect to hyperplanes, is the linear set
\[
L_1 = \{(x,x^q,\Tr(x))_{q^n}:x \in \Fn^\ast\}\subseteq \PG(2,q^n),
\]
where $n>3. $This is scattered, since
\[
(x,x^q,\Tr(x))_{q^n} = (y,y^q,\Tr(y))_{q^n} \Leftrightarrow x/y \in \Fq.
\]
However the line $(0,0,1)^\perp$ meets $L$ in the linear set
\[
 \{(x,x^q,0)_{q^n}:x \in \Fn^\ast, \Tr(x) = 0\},
\]
which is a linear set of rank $n-1>k-1=2$. The subspace $\Omega_1=\langle 1, x^q,\Tr\rangle_{q^n}$ is clearly not disjoint from $\Sigma_{n-3}$, since it contains the point $\langle \Tr\rangle_{q^n}\in \Sigma$.



We see that for $n\geq 4$, the scattered linear set (considered in \cite{BlLa2000})
\[
L_2 = \{(x,x^q,x^{q^2})_{q^n}:x \in \Fn^\ast\}\subseteq \PG(2,q^n),
\]
is both scattered and scattered with respect to hyperplanes, hence $L_2$ corresponds to an $\Fn$-linear MRD code of minimum distance $n-k+1=n-2$. In this case the corresponding MRD code is a Gabidulin code. The subspace $\Omega_2=\langle x,x^q,x^{q^2}\rangle_{q^n}$ is disjoint from $\Sigma_{n-3}$.

For $n=4$, the  linear set 
\[
L_3 = \{(x,x^{q^2},\Tr(x))_{q^4}:x \in \F_{q^4}^\ast\}\subseteq \PG(2,q^4),
\]
is neither scattered nor scattered with respect to hyperplanes. Again, it is easy to see that $\Omega_3=\langle x,x^{q^2},\Tr(x)\rangle_{q^4}$ is not disjoint from $\Sigma{1}$ as it contains $\Tr$.

For $n=4$, every linear set of rank $4$ in $\PG(2,q^4)$ spanning $\PG(2,q^4)$ is equivalent to one these three examples $L_1,L_2,L_3$: this is because the dual of a plane in $\PG(3,q^4)$ is a point, and equivalence classes of points in $\PG(n-1,q^n)$ are precisely the sets $\Omega_i$. The three examples $L_1$, $L_2$, $L_3$ here arise from the duals of a point of rank $3$, $4$, and $2$ respectively.
\end{remark}

\section{MacWilliams identities for rank-metric codes and duality for weight distributions}\label{mac}\label{companion}
\subsection{MacWilliams identities}

We recall the result of Delsarte \cite{Delsarte} extending the classical MacWilliams identities for linear codes to rank metric codes. We state instead the following more convenient recursion from \cite{Ravagnani}. Though in this paper we require only the case $m=n$, we state the more general result.

\begin{theorem}\cite[Corollary 33]{Ravagnani}\label{ravagnani} Let $\C$ be an $\Fq$-linear subspace of $Mat(k\times m,\Fq)$.
Let $A_i$ denote the number of codewords of rank $i$ in $\C$ and let $B_i$ denote the number of codewords of rank $i$ in $\C^\perp$. 
Put
$$a_{\nu}^k=\frac{q^{m\nu}}{|\C|}\sum_{i=0}^{k-\nu}A_i\left[\begin{array}{c}k-i\\ \nu\end{array}\right],$$
where the square brackets denote the Gaussian coefficient.

Then the $B_j$'s are given by the recursive formula 

\begin{align*}
\begin{cases}
B_0&=1\\
B_\nu&=a_{\nu}^k-\sum_{j=0}^{\nu-1}B_j\left[\begin{array}{c}k-j\\ \nu-j\end{array}\right]\\
B_\nu&=0 \quad \mathrm{for}\ \nu>k.
\end{cases}
\end{align*}
\end{theorem}

This theorem tells us that the rank distribution of an $\Fq$-linear code $\C$ determines the rank distribution of the code $\C^\perp$. In Proposition \ref{weightrank}, we have seen that the rank distribution of a code determines the weight  distribution with respect to hyperplanes of the associated linear set. We obtain:


\begin{theorem} The weight distribution of the linear sets in $\L_\C$ determines the weight distribution of the linear sets in $\L_{\C^\perp}$.
\end{theorem}

 In the case that $\Omega(\C)$ is skew from $\Omega(\C)^\perp$, the linear sets $\Sigma/\Omega(\C)^\perp$ and $\Sigma/\Omega(\C)$ can be retrieved from one another:
$\Sigma/\Omega(\C)$ is obtained by projecting $\Sigma$ from $\Omega(\C)$ onto $\Omega(\C)^\perp$, whereas $\Sigma/\Omega(\C)^\perp$ is obtained by projecting $\Sigma$ from $\Omega(\C)^\perp$ onto $\Omega(\C)$. This shows the following:

\begin{corollary} Let $\C=\la f_1,f_2,\ldots,f_k\ra$. To switch between $L_{f_1,f_2,\ldots,f_k}\in \L_\C$ and  $L_{\la f_1,f_2,\ldots,f_k\ra_{q^n}^\perp}\in \L_{\C^\perp}$ in the case that $\Omega(\C)$ is skew from $\Omega(\C)^\perp$ just requires a switch between the subspace from which we are projecting and the subspace we are projecting onto.
\end{corollary}

\begin{remark}
If $n=4$ and $k=2$, then $L_{f_1,f_2}$ and $L_{{ \la f_1,f_2\ra_{q^n}}^\perp}$ are both equivalence classes of linear sets of rank $4$ in $\PG(1,q^4)$. Such linear sets have been classified in \cite{Polverino}. 

Note that if $U$ is a four-dimensional $\Fq$-subspace of $V(2,q^4)$, given an $\Fq$-bilinear form on $V(2,q^4)$ we can define $U^\perp$, and hence can define the dual $L(U^\perp)$ of a linear set $L(U)$, as in \cite{olga}. However, if $U$ is such that  $L(U)\in L_\C$, it is not necessarily true that $L(U^\perp)\in L_{\C^\perp}$. Thus mapping $L_{f_1,f_2,\ldots,f_k}$ onto  $L_{\la  f_1,f_2, \ldots,f_k\ra^\perp}$  defines is a new operation {\sout{on equivalence classes}}, different from the dual operation.
\end{remark}

\subsection{An application for linear sets on a line}
If a linear set spans a line, then its weight distribution with respect to hyperplanes is determined by the weights of the points in the set. In this subsection, we will use this fact to deduce one of the MacWilliams identities in a geometric way. 

Studying the weight distribution of linear sets on a line was partially motivated by a problem that arose during the study of KM arcs. It was shown in \cite{KMarcs} that translation KM-arcs of type $2^i$ in $\PG(2,2^h)$ are equivalent to $i$-clubs in $\PG(1,2^h)$. These $i$-clubs are $\F_2$-linear sets that have exactly one point of weight $i$ and all others of weight $1$. The existence of $i$-clubs in $\PG(1,2^h)$ is known only for a few parameter values. The case $i=2$ is of particular interest as it is known that there are (small) values of $h$ for which no $2$-clubs in $\PG(1,2^h)$ exist; this is a topic requiring further investigation. In Theorem \ref{wt2} we give an equivalent condition for the existence of linear sets with a fixed number $N$ of points of weight $2$ and all others of weight $1$.

\begin{definition}A {\em proper} linear set is a linear set which contains more than one point.
\end{definition}
\begin{proposition}\label{BW} Let $L$ be a proper linear set of rank $n$, $n\geq 3$,  in $\PG(1,q^n)$ that is obtained by projecting a subgeometry $\Omega=\PG(n-1,q)$ from a subspace $\Pi=\PG(n-3,q^n)$ contained in $\PG(n-1,q^n)$. Let $R_2$ be the number of points of rank $2$ in $\Pi$ and let $W_i$ be the number of points of weight $i$ in $L$, then 
$$R_2=\sum_{i=2}^{n-1}W_i\left[\begin{array}{c}i\\2\end{array}\right].$$
\end{proposition}

\begin{proof}If a point $P$ of $L$ has weight $i$, then $\langle P, \Pi\rangle$ meets $\Omega$ in an $(i-1)$-space $\tau$ of $\Omega$ (see Proposition \ref{weightproj}). The extension of this space $\tau$ to a subspace of $\PG(n-1,q^n)$ is denoted by $\bar{\tau}$. Let $\ell$ be a line of $\tau$ and let $\bar{\ell}$ be its extension (which lies in $\bar{\tau}$), then $\bar{\ell}$ meets the hyperplane $\Pi$ of $\langle P,\Pi\rangle$ in a point $Q$.  Note that there are $\left[\begin{array}{c}i\\2\end{array}\right]$ lines in $\tau$. 

This point $Q$ lies on an extended line of $\Omega$ and hence, has rank $2$. Now, suppose we have a different extended line $\bar{\ell'}$ that meets $\Pi$ in a point $Q'$. If $Q$ and $Q'$ would coincide, then $Q=Q'$ is the intersection point of $\bar{\ell}$ and $\bar{\ell'}$. But this implies that $\ell$ and $\ell'$ are contained in a plane, and hence, that the intersection point $Q$ is contained in $\Omega$. Since $Q$ is a point of $\Pi$, this in turn implies that $\Pi$ is not skew from $\Omega$, a contradiction. So we may conclude that all lines $\ell$ obtained in this fashion give rise to different points of rank $2$. As a point of weight $i$ gives rise to $\left[\begin{array}{c}i\\2\end{array}\right]$ points of rank $2$, we find in total
$\sum_{i=2}^{n-1}W_i\left[\begin{array}{c}i\\2\end{array}\right]$ points of rank $2$. Note that $W_n=0$ since the existence of a point of weight $n$ in $L$ would imply that $L$ equals that point, a contradiction since $L$ is a proper linear set.

\end{proof}
\begin{corollary}\label{wt1-2}Let $L$ be a proper linear set of rank $n$, $n\geq 3$,  in $\PG(1,q^n)$ that is obtained by projecting a subgeometry $\Omega=\PG(n-1,q)$ from a subspace $\Pi=\PG(n-3,q^n)$ contained in $\PG(n-1,q^n)$.
If $L$ has only points of weight $1$ and $2$, the number of points of weight $2$ in $L$ is the number of points of rank $2$ in $\Pi$.
\end{corollary}

We can use Corollary \ref{wt1-2} to obtain the following geometric construction for linear sets having all but a few points of weight $1$ on a line.
\begin{theorem}\label{wt2}
There exists a linear set of rank $n$ in $\PG(1,q^n)$ containing $N$ points of weight $2$ and all other points of weight at most $1$ if and only if there exists a subspace of co-dimension two in $\PG(n-1,q^n)$ disjoint from $\Sigma$ and meeting $\Sigma_2$ in precisely $N$ points.
\end{theorem}
Unfortunately, it does not appear to be easy to determine the possibilities for the intersection of a subspace with $\Sigma_2$. However, this does provide an alternative approach which may be of benefit, as $\Sigma_2$ is a more convenient variety to work with than $\Sigma_{n-2}$.

\subsection{Using MacWilliams identities to prove Proposition \ref{BW}}

\begin{lemma} Let $\C=\la f_1,f_2\ra$ such that $L\in \L_\C$ is a proper linear set of rank $n$ in $\PG(1,q^n)$, and let $A_i$ denote the number of vectors of rank $i$ in $\C$, then we have that $$\sum_{i=1}^{n-1}A_i\left[\begin{array}{c}n-i\\ 1\end{array}\right]=(q^n-1)\left[\begin{array}{c}n\\ 1\end{array}\right].$$
\end{lemma}

\begin{proof}

Consider the representation of $L$ as the set of elements of a Desarguesian spread meeting an $(n-1)$-dimensional projective space $\pi$. The elements of $\D$ intersecting $\pi$ correspond to the points of $L$. More precisely, we have that a point of $L$ has weight $i$ if and only if the corresponding spread element meets $\pi$ in an $(i-1)$-dimension projective space. The $(i-1)$-spaces corresponding to the points of $L$ form a partition of $\pi$, so we have
$$\sum_{i=1}^{n}W_i\left[\begin{array}{c}i\\ 1\end{array}\right]=\left[\begin{array}{c}n\\ 1\end{array}\right],$$
where $W_i$ is the number of points of weight $i$ in $L$.
First note that $W_n=0$ since $L$ is proper. Secondly, we have seen that a point of weight $i$ in $L$ corresponds to a point of rank $n-i$ in $\Omega(\C)$. 
So the number of points with weight $W_i$ is the number of points of rank $n-i$ in $\Omega(\C)$. This number in turns equals $\frac{A_{n-i}}{q^n-1}$ as every point gives rise to $(q^n-1)$ vectors in $\C$. So we obtain that 
$$\sum_{i=1}^{n-1}\frac{A_{n-i}}{q^n-1}\left[\begin{array}{c}i\\ 1\end{array}\right]=\left[\begin{array}{c}n\\ 1\end{array}\right],$$ or equivalently
 $$\sum_{i=1}^{n-1}A_i\left[\begin{array}{c}n-i\\ 1\end{array}\right]=(q^n-1)\left[\begin{array}{c}n\\ 1\end{array}\right].$$
\end{proof}

\begin{corollary} \label{An-1}$A_{n-1}=\frac{(q^n-1)^2}{q-1}-\sum_{i=0}^{n-2}A_i\left[\begin{array}{c}n-i\\ 1\end{array}\right]+\left[\begin{array}{c}n\\ 1\end{array}\right]$
\end{corollary}

\begin{lemma} Let $A_i$ be the number of vectors of rank $i$ in $\C=\la f_1,f_2\ra$ and let $B_i$ denote the number of vectors of rank $i$ in $\C^\perp$. Then $B_1$=0.
\end{lemma}
\begin{proof} We use Ravagnani's formulae from Theorem \ref{ravagnani} with $k=m=n$, $|\C|=q^{2n}$ and $\nu=1$. We find
$$B_1=\frac{q^n}{q^{2n}}\left(\sum_{i=0}^{n-1}A_i\left[\begin{array}{c}n-i\\ 1\end{array}\right]\right)-B_0\left[\begin{array}{c}n\\ 1\end{array}\right],$$or
$$B_1=\frac{1}{q^{n}}\left(\sum_{i=0}^{n-3}A_i\left[\begin{array}{c}n-i\\ 1\end{array}\right]\right)+\frac{1}{q^{n}}A_{n-2}\left[\begin{array}{c}2\\ 1\end{array}\right]+\frac{1}{q^{n}}A_{n-1}\left[\begin{array}{c}1\\ 1\end{array}\right]-\left[\begin{array}{c}n\\ 1\end{array}\right].$$

We sustitute $A_{n-1}$ for its value found in Corollary \ref{An-1}, that is 
$$A_{n-1}=\frac{(q^n-1)^2}{q-1}-\sum_{i=0}^{n-3}A_i\left[\begin{array}{c}n-i\\ 1\end{array}\right]-A_{n-2}\left[\begin{array}{c}2\\ 1\end{array}\right]+\left[\begin{array}{c}n\\ 1\end{array}\right].$$

Plugging this in the equation for $B_1$, we find that $B_1=0$.

\end{proof}

Now we will only use Theorem \ref{ravagnani} to prove:
\begin{lemma} \label{la}If $A_i$ is the rank distribution of $\C=\la f_1,f_2\ra$, then $$B_2=\sum_{i=1}^{n-2}A_i\left[\begin{array}{c}n-i\\ 2\end{array}\right].$$
\end{lemma}
\begin{proof}
We have that $B_2=a_2^n-B_0\left[\begin{array}{c}n\\ 2\end{array}\right]-B_1\left[\begin{array}{c}n-1\\ 1\end{array}\right]$. Now we have seen that $B_1=0$ in the previous lemma.
Further, we have that $a_2^n=\sum_{i=0}^{n-2}A_i\left[\begin{array}{c}n-i\\ 2\end{array}\right], $ so
$$B_2=\sum_{i=0}^{n-2}A_i\left[\begin{array}{c}n-i\\ 2\end{array}\right]-\left[\begin{array}{c}n\\ 2\end{array}\right]=\sum_{i=1}^{n-2}A_i\left[\begin{array}{c}n-i\\ 2\end{array}\right].$$
\end{proof}

\begin{remark}
We have seen in Proposition \ref{BW} that $$R_2=\sum_{i=2}^{k-1}W_i\left[\begin{array}{c}i\\2\end{array}\right].$$ To retrieve this result, put $R_2=A_2/(q^n-1)$, $W_i=A_{n-i}/(q^n-1)$ and $k=n$ in Lemma \ref{la}.

\end{remark}

\begin{remark} The ideas developed in this section can be used to explicitely compute the weight distributions of the companions of linear sets with prescribed weight distributions. For example, if $L_{f_1,f_2}$ is a linear set of rank $5$ in $\PG(1,q^5)$ that has 1 point of weight $2$ and all others of weight $1$, then the rank distribution of $\la f_1,f_2\ra$ is as follows:

\begin{align*}
A_0&=1\\
A_1&=0\\
A_2&=0\\
A_3&=q^5-1\\
A_4&=(q^5-1)\frac{q^n-q^2}{q-1}\\
A_5&=(q^5-1)\left(q^n-\frac{q^n-q^2}{q-1}\right)
\end{align*}
and we have that the rank distribution of $\Omega(\C)$ is given by the vector $$rk(\Omega(\C))=\left(0,0,1,\frac{q^n-q^2}{q-1},\left(q^n-\frac{q^n-q^2}{q-1}\right)\right).$$ This says that on the line $\la f_1,f_2\ran$, there is one point of rank $3$, $\frac{q^n-q^2}{q-1}$ points of rank $4$ and the remaining $q^n-\frac{q^n-q^2}{q-1}$ ones are of rank $5$.

Applying Theorem \ref{ravagnani} gives us
\begin{align*}
B_0&=1\\
B_1&= 0\\
B_2&=q^5-1\\
B_3&=(q^5-1)(q^6 + q^5 + 2q^4 + 2q^3 + q^2)\\
B_4&=(q^5-1)(q^9 + q^8 - q^6 - 2q^5
    - 3q^4 - 2q^3 - q^2)\\
B_5&=(q^5-1) ( q^{10} - q^9 - q^8 + 2q^5 + q^4).
\end{align*}
and we have that the rank distribution of $\Omega(\C^\bot)$ is given by the vector $$rk(\Omega(\C^\bot))=(0,
1,
q^6 + q^5 + 2q^4 + 2q^3 + q^2,
q^9 + q^8 - q^6 - 2q^5
    - 3q^4 - 2q^3 - q^2,
 q^{10} - q^9 - q^8 + 2q^5 + q^4).$$
We conclude that  $L_{h_1,h_2,h_3}$ with $\la h_1,h_2,h_3\ran=\la f_1,f_2\ran^\bot$ is a linear set in $\PG(2,q^5)$ with the following weight distribution with respect to hyperplanes; i.e., lines:
$$w_{L_{h_1,h_2,h_3}}=( q^{10} - q^9 - q^8 + 2q^5 + q^4,q^9 + q^8 - q^6 - 2q^5
    - 3q^4 - 2q^3 - q^2,q^6 + q^5 + 2q^4 + 2q^3 + q^2,1,0).$$
We find a unique line that has weight $3$ with respect to $L_{h_1,h_2,h_3}$ while all other lines have weight $1$ or $2$ with respect to $L_{h_1,h_2,h_3}$.

\end{remark}


%
%
%
%
%


\end{document}